\documentclass[11pt,a4paper,twoside]{article}

\usepackage{authblk}
\usepackage[T1]{fontenc} 
\usepackage{lmodern}   

\usepackage{textcomp}
\usepackage{setspace}
\usepackage[top=3cm,bottom=3cm, left=3cm,right=3cm,footskip=1.5cm]{geometry}

\usepackage{mathtools}
\usepackage{amsthm,amssymb,amscd,amsfonts}
\usepackage{mathrsfs}
\usepackage{tikz}
\usepackage{tikz-cd}
\usepackage{aliascnt} 

\usepackage{graphicx}
\usepackage{float}
\usepackage{tabularx}
\usepackage{ltablex}
\usepackage{booktabs}
\usepackage[table]{xcolor}
\usepackage{adjustbox}
\usepackage{enumitem}
\usepackage{pdfpages}
\usepackage{appendix}

\setlist[itemize]{noitemsep}
\setlist[enumerate]{itemsep=0mm}

\definecolor{persianblue}{rgb}{0.11, 0.22, 0.73}
\definecolor{persiangreen}{rgb}{0.0, 0.65, 0.58}
\definecolor{antiquefuchsia}{rgb}{0.57, 0.36, 0.51}

\usepackage[defernumbers=true,maxnames=6,
backend=biber,style=numeric,giveninits=true,sortcites=true,
sorting=nty,doi=true,isbn=true,url=true,safeinputenc]{biblatex}
\usepackage[bookmarks=false,hidelinks]{hyperref}
\hypersetup{
	colorlinks   = true,
	urlcolor     = antiquefuchsia,
	linkcolor    = persianblue,
	citecolor    = persianblue
}

\usepackage[nameinlink,capitalize]{cleveref}

\theoremstyle{definition}

\newtheorem{theorem}{Theorem}[section]
\crefname{theorem}{Theorem}{Theorems}

\newcommand{\newaliastheorem}[2]{
  \newaliascnt{#1}{theorem}  
  \newtheorem{#1}[#1]{#2}     
  \aliascntresetthe{#1}       
  \crefname{#1}{#2}{#2s}      
}

\newaliastheorem{lemma}{Lemma}
\newaliastheorem{proposition}{Proposition}
\newaliastheorem{corollary}{Corollary}
\newaliastheorem{claim}{Claim}
\newaliastheorem{subclaim}{Sub-claim}
\newaliastheorem{fact}{Fact}
\newaliastheorem{observation}{Observation}
\newaliastheorem{definition}{Definition}
\newaliastheorem{remark}{Remark}
\newaliastheorem{question}{Question}
\newaliastheorem{example}{Example}
\newaliastheorem{notation}{Notation}
\newaliastheorem{convention}{Convention}
\newaliastheorem{op}{Open Problem}
\newtheorem*{mainthm}{Main Theorem}
\newtheorem*{singlequestion}{Question}

\newtheorem{namedlemma}{}

\newcommand{\dlevcol}[2]{{\dot{\textrm{Col}}(#1,{<}#2)}}
\newcommand{\cp}[1]{\mathrm{cp}(#1)} 
\newcommand{\ult}[2]{{\mathrm{Ult}(#1,#2)}} %
\newcommand{\rank}[1]{\mathrm{rank}(#1)}
\newcommand{\forces}[2]{\Vdash_{#1}\text{``}#2\text{''}} 
\newcommand{\df}{\coloneqq} 
\newcommand{\dom}[1]{\mathrm{dom}(#1)} 
\newcommand{\mbb}[1]{\mathbb{#1}} 
\newcommand{\mcl}[1]{\mathcal{#1}} 
\newcommand{\ma}{\mathrm{MA}} 

\newcommand{\rc}{\mathrm{RC}} 
\newcommand{\col}[2]{\mathrm{Col}(#1,#2)} 
\newcommand{\dcol}[2]{\dot{\mathrm{Col}}(#1,#2)} 
\newcommand{\levcol}[2]{\mathrm{Col}(#1,{<}#2)} 

\newcommand{\email}[1]{\href{mailto:#1}{\texttt{#1}}}
\begin{document}
	
\title{Todorcevic's Problem on Rado's Conjecture}
\author[1]{Monroe Eskew\footnote{\email{monroe.eskew@univie.ac.at}}}
\author[2]{Rahman Mohammadpour\footnote{\email{rahmanmohammadpour@gmail.com}}}

\affil[1]{\small Kurt G{\"{o}}del Research Center, University of Vienna, Kolingasse 14-16, 1090 Vienna, Austria}
\affil[2]{\small Institute of Mathematics of the Polish Academy of Sciences, Jana i J\k{e}drzeja \'{S}niadeckich 8, 00-656 Warsaw, Poland}

	\date{}
\maketitle

\begin{abstract}
In his Mostowski lecture in Wroc{\l}aw in 2024, Stevo Todor\v{c}evi\'c asked whether it is consistent that Rado's Conjecture holds at two successive cardinals. We show that it is consistent that Rado's Conjecture holds at all regular cardinals.
\end{abstract}
\section{Introduction}
	\emph{Rado's Conjecture} ($\rc$) stands as a distinguished reflection principle in combinatorial set theory, notable for implying a wide range of compactness principles while contradicting  $\ma({\omega_1})$. To see this, note that $\textsf{ZFC}$ 
    proves the existence of a non-special tree $T$ of height $\omega_1$ without cofinal branches-- see for example
    \cite[Section 9]{Todorcevic-handbook}, so by $\rc$, there  exists a non-special tree $S\subseteq T$ of size and height $\omega_1$. In particular, $S$ has no cofinal branches.  On the other hand, by \cite{BMR,Bau83incollectionIter}, $\ma(\omega_1)$ implies every tree of height and size $\omega_1$ without cofinal branches is special, and hence a contradiction.
    The original formulation of Rado's Conjecture states that if an intersection graph (cf.\ \cite{Tod83articleOn-a}) is not countably chromatic, then there is an $\omega_1$-sized subgraph which is not countably chromatic. 
    While contradicting $\ma({\omega_1})$, it implies several strong consequences of Martin's Maximum like the Singular Cardinal Hypothesis~\cite{Tod93incollectionConj}, the Semistationary Reflection Principle~\cite{Doe13articleRado}, and the presaturation of the nonstationary ideal on $\omega_1$ \cite{Fen99articleRado}, among others---see, for example, Todor\v{c}evi\'c's survey \cite{Tod11articleComb}.  
    In~\cite{Tod83articleOn-a}, Todor\v{c}evi\'c introduced Rado's Conjecture, referring to a conjecture stated by Rado in~\cite{Rad81articleTheo}, and established the consistency of the conjecture and its generalizations to higher cardinals.    Among   many foundational contributions to the subject, Todor\v{c}evi\'c also provided a rather advantageous reformulation of the conjecture in terms of trees, which we will take as the official formulation. However, let us first recall the definition of a special tree.

\begin{definition}[Todor\v{c}evi\'c \cite{Tod85articlePart}]
    Suppose $\kappa$ is a regular cardinal and $T$ is a tree of height $\kappa$. We say that $T$ is \emph{special} if there exists a function $f : T \to T$ such that:
    \begin{enumerate}
        \item For every $t \in T$, $f(t) <_T t$.
        \item For every $t \in T$, $f^{-1}[\{t\}]$ is the union of fewer than $\kappa$-many antichains.
    \end{enumerate}
\end{definition}

It is a well-known theorem of Todor\v{c}evi\'c \cite{Tod85articlePart}  that if the height of $T$ is a successor cardinal $\kappa^+$, $T$ is special if and only if $T$ can be decomposed into $\kappa$-many antichains, hence the above definition coincides with the classical notion of a special tree for trees whose height are successor cardinals.
    
    \begin{definition}[Generalized Rado's Conjecture \cite{Tod83articleOn-a}]
Let $\kappa$ be an uncountable cardinal.
The principle $\rc(\kappa)$ states that every non-special tree of height $\kappa$ has a non-special subtree of size and height $\kappa$.
\end{definition}
  
It is worth noting that our notation is different from that of Todor\v{c}evi\'c. His $\rc(\kappa)$ corresponds to our $\rc(\kappa^+)$. Nevertheless, $\rc$ always stands for the original Rado's Conjecture, denoted $\rc(\omega)$ by Todor\v{c}evi\'c and $\rc(\omega_1)$ by us. We state Todor\v{c}evi\'c's consistency result using our notation.

\begin{theorem}[Todor\v{c}evi\'c~{\cite{Tod83articleOn-a}}]\label{Todorcevic-theorem}
    Let $\kappa$ be a cardinal below a supercompact cardinal $\lambda$. Then $\rc(\kappa^+)$ holds in $V^{\levcol{\kappa^+}{\lambda}}$.
\end{theorem}

 The following problem was mentioned by Todor\v{c}evi\'c during his Mostowski lecture in Wroc{\l}aw in 2024.
	
	\begin{singlequestion}[Todor\v{c}evi\'c \cite{Todorc-slide}]
		Is it consistent that $\rc(\kappa^+)$ holds, for every  cardinal $\kappa\geq\aleph_0$?
	\end{singlequestion}

    It is worth mentioning that it was not known, as mentioned verbally by Todor\v{c}evi\'c, if for some cardinal $\kappa$, $\rc(\kappa^+)$ and $\rc(\kappa^{++})$ hold simultaneously.
In this paper, we answer his  question in a slightly more general way:
	
	\begin{mainthm}
     Assume that there exists a proper class of supercompact cardinals. Then there is a class forcing extension in which $\rc(\lambda)$ holds for every regular cardinal $\lambda$.
	\end{mainthm}

We provide the necessary materials in the next section. We will then establish a connection between ideals and certain variants of Rado's Conjecture in  \cref{sec:ideal and rc}. This serves as a toy version of our argument in \cref{sec:grc}. where  we will prove  Main Theorem.

 \section{Basics}\label{basics}
We use standard conventions in set theory; the reader can find the relevant material in \cite{Jec03bookSet-}. More specifically, our conventions in forcing theory are standard: stronger conditions are smaller in ${\leq}$. A forcing is $\kappa$-closed, for an uncountable cardinal $\kappa$, if every descending sequence of conditions of length less than $\kappa$ has a lower bound. Let $X$ be a set and let $\kappa\leq |X|$. We let $\mcl P_\kappa(X)$ denote the collection of subsets of $X$ that are of size less than $\kappa$. 

	Two forcing notions $\mbb P,\mbb Q$ are equivalent if for every $V$-generic filter $G_{\mbb P}\subseteq \mbb P$, there is a $V$-generic filter $G_{\mbb Q}\subseteq\mbb Q$ such that $V[G_{\mbb P}]=V[G_{\mbb Q}]$, and vice versa. If $\mbb P,\mbb Q$ are equivalent, we write $\mbb P\sim \mbb Q$.
    So $\sim$ is an equivalence relation.

\begin{lemmaname}{McAloon's Lemma}\label{McAloon-Lemma}
Let $\mbb P$ be a $\kappa$-closed forcing. Assume that $\mbb P$ forces that $|\mbb P|=\kappa$. Then $\mbb P$ contains a dense subset isomorphic to $\col{\kappa}{|\mbb P|}$, so $\mbb P\sim \col{\kappa}{|\mbb P|}$.
\end{lemmaname}

\begin{lemma}[Todor\v{c}evi\'c~{\cite{Tod83articleOn-a}}]\label{Todorcevic-lemma}
    Let $\kappa$ be a regular cardinal, and let $T$ be a non-special tree of height $\kappa$. Let $\mbb P$ be a $\kappa$-closed forcing. Then $\mbb P$ does not specialize $T$.
\end{lemma}
\begin{proof}
    Assume the contrary, and let $\dot f$ be a $\mbb P$-name for a specializing function. For each $t \in T$, let $\dot g_t$ be a name for a function with domain $\dot f^{-1}[\{ t \}]$ and codomain some $\dot\mu_t<\kappa$, such that $\dot g_t$ is forced to be injective on chains.

    We build an order-reversing map $t \mapsto p_t$ by induction on the rank of the nodes in $T$. For all minimal nodes $t$ of $T$, let $p_t = {1}_{\mbb P}$. Given the map up to rank $\alpha$, first suppose $\alpha = \beta+1$. If $\rank{t}= \alpha$, and $s <_T t$ has rank $\beta$, let $p_t \leq p_s$ decide the value of $\dot\mu_t$, decide $\dot f(t)$ as some $r <_T t$, and decide $\dot g_r(t)$. If $\alpha$ is a limit ordinal, then for every node $t$ of rank $\alpha$, first take a lower bound $p'_t$ to $\{ p_s : s <_T t \}$ using $\kappa$-closure, and then extend $p'_t$ to $p_t$ to make the decisions as in the successor case.

    Let $F : T \to T$ be defined by $F(t) = r$ if and only if $p_t \forces{}{\dot f(t) = r}$; note that $F$ is a regressive function. If $F(t) = r$, define $G(t) = \alpha$ if and only if $p_t \forces{}{\dot g_r(t) = \alpha}$. Let $\nu_t$ be the cardinal decided by $p_t$ as the value of $\dot\mu_t$. If $t \leq_T s$ and $F(t) = F(s) = r$, then $p_s \leq p_t \leq p_r$. If $\alpha_t,\alpha_s$ are such that $p_t \forces{}{\dot g_r(t) = \alpha_t}$ and $p_s \forces{}{\dot g_r(s) = \alpha_s}$, then $\alpha_t,\alpha_s < \nu_r$. Also, $p_s \forces{}{t,s \in \dom{\dot g_r} \wedge \dot g_r(t) \not= \dot g_r(s)}$, so $\alpha_t \not= \alpha_s$. Thus $F,G$ witness that $T$ is special.
\end{proof}

\section{Rado's Conjecture and its variants}\label{sec:ideal and rc}

\begin{definition}
    Let $\kappa$ be an infinite regular cardinal. Let $\mu\geq\lambda\geq\kappa^{+}$ be cardinals.
    
    \begin{enumerate}
        \item   The \emph{$(\kappa,\lambda,\mu)$-Rado's Conjecture}, denoted by $\rc(\kappa,\lambda,\mu)$, states that every non-special tree of height $\kappa$ and size at most $\mu$ has a non-special subtree of height $\kappa$ and size less than $\lambda$.
        \item 
   The \emph{$(\kappa,\mu)$-Rado's Conjecture}, denoted by $\rc(\kappa,\mu)$, stands for $\rc(\kappa,\kappa^{+},\mu)$. 
    \end{enumerate}
\end{definition}
Note that the familiar $\kappa$-Rado's Conjecture, denoted by $\rc(\kappa)$, states that $\rc(\kappa,\mu)$ holds for every cardinal $\mu\geq\kappa^{+}$.

\begin{theorem}
    Let $\kappa<\lambda$ be infinite cardinals with $\kappa$ regular. Let $\mcl I$ be a normal $\kappa^+$-complete ideal on $\mcl P_{\kappa^+}(\lambda)$. Assume that 
     $\mcl P(\mcl P_{\kappa^+}(\lambda))/\mcl I$ has a $\kappa$-closed dense set. Then $\rc(\kappa,\lambda)$ holds.
\end{theorem}
\begin{proof}
    Let $T$ be a non-special tree of height $\kappa$ and size $\leq\lambda$. Let $G$ be a $V$-generic filter on $\mcl P(\mcl P_{\kappa^+}(\lambda))/\mcl I$.  It is well-known that if such a forcing is equivalent to a countably closed poset, then the ideal $\mcl I$ is precipitous, i.e.\ the generic ultrapower is always well-founded (see \cite[Lemma 22.19]{Jec03bookSet-}). Let $j:V\to M\cong \ult{V}{G}\subseteq V[G]$ be the generic ultrapower embedding induced by $G$, where $M$ is transitive.
    Note that $\cp{j}=\kappa^+$ and $j[\lambda] \in M$.
    By the assumption, $V[G]$ is an extension of $V$ by a $\kappa$-closed forcing.
    So \cref{Todorcevic-lemma} implies that $T$ remains non-special in $V[G]$. In particular, $T^*\df j[T]\cong T$ is a non-special subtree of $j(T)$. 
    We may assume that $T$ is coded as a relation on $\lambda$, so $T^* \in M$.  Since being special is upwards-absolute, $T^*$ is non-special in $M$ as well.  Note that $j(\kappa^+)>\lambda$.
    Thus $M$ knows that $j(T)$ has a non-special subtree of height $\kappa$ and size $<j(\kappa^+)$. So, by elementarity, $T$ has a non-special subtree of height and size $\kappa$.
\end{proof}

The above result implies a model in which $\rc(\kappa)$ holds for all regular cardinals $\kappa$ can be obtained by arranging that for every regular $\kappa$, there are unboundedly many $\lambda$ such that there is a $\kappa^+$-complete normal ideal $\mcl I$ on $\mcl P_{\kappa^+}(\lambda)$ such that $\mcl P(\mcl P_{\kappa^+}(\lambda))/\mcl I$ contains a $\kappa$-closed dense set.   However, this may not be the optimal way to get the consistency of $\rc(\kappa)$ holding everywhere.  On the one hand, if $\mu$ is regular and $\kappa>\mu$ is supercompact, then forcing with $\levcol{\mu}{\kappa}$ yields a model in which the desired ideals on $\mcl P_{\kappa}(\lambda)$ exist for all $\lambda\geq\kappa$ (see \cite{foreman}).  If we want to arrange this situation for multiple values of $\kappa$, we might try then forcing with $\levcol{\kappa}{\kappa'}$ for some supercompact $\kappa'>\kappa$.  But as this adds many subsets of $\kappa$, it is not clear whether this preserves even that there is an ideal $\mcl I$ on $\kappa$ such that $\mcl P(\kappa)/ \mcl I$ has a $\mu$-closed dense set.  As shown in \cite{fms}, iterating the L\'evy collapse between (partially) supercompact cardinals does suffice to get a model in which every successor cardinal carries a \emph{precipitious} ideal, but the argument does not seem to suffice for getting quotient algebras that are equivalent to highly closed forcing.

An alternate forcing argument as in \cite{morerigid}  obtains a model containing the desired ideals (with much stronger properties) for all pairs of regular cardinals $\kappa<\lambda$, starting from a huge cardinal.\footnote{That paper was focused on rigidity properties of the quotient algebra. To get quotient algebras just with the desired closure properties, a simpler forcing as in \cite{shioya} can be used.}  However, such strong large cardinal assumptions are not needed if we focus on preserving $\rc(\kappa)$ itself rather than hypotheses about ideals.  That is the approach of the next section.

\section{Global Rado's Conjecture}\label{sec:grc}

In this section, we prove our main theorem.

\begin{theorem}\label{mainthm}
    Assume that there exists a proper class of supercompact cardinals. Then there is a class forcing extension in which $\rc(\lambda)$ holds for every regular cardinal $\lambda$.
\end{theorem}
\begin{proof}
    Let $C = \langle \kappa_\alpha : \alpha \in \Omega \rangle$ enumerate the closure of the class of supercompact cardinals in increasing order, where $\Omega$ is the class of all ordinals.
    We shall define an Easton support iteration. However, 
let us first partition $\Omega$ into two classes, $\Omega_1,\Omega_2$ by declaring that an ordinal $\alpha$ belongs to $\Omega_1$ if and only if $\kappa_\alpha$ is regular. Now define an Easton-support iteration $\mbb P\df \langle \mbb{P}_\beta,\dot{\mbb Q}_\alpha: \alpha<\beta\in\Omega\rangle$, where $\dot{\mbb Q}_\alpha$ is a $\mbb P_\alpha$-name for a Levy collapse by letting

\begin{equation*}
    \mbb P_{\alpha+1} \df \begin{cases} 
     \levcol{\omega_1}{\kappa_
    0} & \text{if } \alpha=-1\\
   \mbb{P}_\alpha \ast \dlevcol{\kappa_\alpha}{\kappa_{\alpha+1}} & \text{if } \alpha\in\Omega_1\\
 \mbb{P}_\alpha \ast \dlevcol{\kappa_\alpha^+}{\kappa_{\alpha+1}}& \text{if } \alpha\in\Omega_2\\
   \end{cases}
\end{equation*}

    It is routine to verify that $\mbb P$  turns $\kappa_n$ into $\aleph_{n+2}$ for $n<\omega$, preserves inaccessible limits and successors of singular limits of $C$, and forces the class of successors of regular uncountable cardinals to coincide with the successor points of $C$.\footnote{We could start with $\mbb{P}_1$ as the Mitchell forcing to obtain a model with $2^\omega = \omega_2$;  for a proof of the fact that $\rc$ holds in Mitchell's model for the tree property on $\omega_2$ see \cite{Zha20articleRado}.}

    Let $G_\Omega$ denote a generic filter for this class forcing, and let $G_\alpha$ denote $G_\Omega \cap \mbb{P}_\alpha$ for every ordinal $\alpha$.
    Suppose that in the forcing extension $V[G_{\Omega}]$, $\lambda$ is a regular cardinal and $T$ is a non-special tree of height $\lambda$. There is some $\alpha$ such that $(\lambda^+)^{V[G_{\Omega}]} = \kappa_{\alpha+1}$, and there is some $\beta$ such that $T \in V[G_\beta]$. Without loss of generality, we may assume that $\beta\geq\alpha$. Note that if $T$ has a non-special subtree $T'$ of height and size $\lambda$ in $V[G_\beta]$, then this remains true in $V[G_\Omega]$ because the tail forcing $\mbb{P}_\Omega/G_\beta$ is $\lambda^+$-closed, which implies it cannot add any functions from $T'$ to itself.

    Therefore, our focus will be on showing that $T$ has a non-special subtree of height and size $\lambda$ in $V[G_\beta]$. For this, we need to prepare the ground. So let $\dot T$ be a $\mbb{P}_\beta$-name for $T$.
   Without loss of generality, we may assume that $\dot{T}$ is a relation on an ordinal $\gamma$.
Fix $\delta > \kappa_{\beta+1},\gamma$. Note that $\mbb{P}_\beta$ is of size less than $\delta$ and that $\forces{\mbb{P}_\beta}{|\dot T| < \delta}$.

Working in $V[G_\alpha]$, the forcing $\mbb{P}_{\alpha+1}$ has the form $\mbb{P}_\alpha \ast \dlevcol{\lambda}{\kappa_{\alpha+1}}$, 
where we identify $\lambda$ with the cardinal $\kappa_\alpha$ if $\alpha\in\Omega_1$, or with $\kappa_\alpha^+$ otherwise. Note that $\kappa_{\alpha+1}$ is still supercompact by standard arguments, so let us fix a normal $\kappa_{\alpha+1}$-complete ultrafilter $\mcl{U}$  
on $\mcl{P}_{\kappa_{\alpha+1}}(\delta)$. Let $\theta \df j_{\mcl{U}}(\kappa_{\alpha+1})$, where $j_{\mathcal U}$ is the ultrapower embedding  induced by  $\mathcal{U}$ over $V[G_\alpha]$.

Now, we momentarily move to a generic extension of $V[G_\alpha]$ by $\mbb P_{\alpha+1}/G_\alpha$. Let us abuse notation and denote it by $V[G_{\alpha+1}]$, and let $\mbb{Q} \df\mbb{P}_\beta / G_{\alpha+1}$. Also, let $\eta \df |\mbb{P}_\beta| = |\mbb{Q}|^{V[G_{\alpha+1}]}$. 
Note that $\mbb{Q}$ is $\kappa_{\alpha+1}$-closed in $V[G_{\alpha+1}]$; so by \cref{McAloon-Lemma}, $\col{\lambda}{\eta}$ is isomorphic to a dense subset of $ \mbb{Q} \times \col{\lambda}{\eta}$ in $V[G_{\alpha+1}]$.
Therefore, in $V[G_\alpha]$, we have the following forcing equivalences:
 \begin{align*}
        \levcol{\lambda}{\theta} &\sim \levcol{\lambda}{\kappa_{\alpha+1}} \times \levcol{\lambda}{[\kappa_{\alpha+1},\theta)} \\
                                &\sim  \levcol{\lambda}{\kappa_{\alpha+1}} \ast \dlevcol{\lambda}{[\kappa_{\alpha+1},\theta)} \\
                                &\sim  \levcol{\lambda}{\kappa_{\alpha+1}} \ast \left( \dcol{\lambda}{\eta} \times \dlevcol{\lambda}{[\eta,\theta)} \right)\\
                                &\sim  \levcol{\lambda}{\kappa_{\alpha+1}} \ast \left( \dot{\mbb{Q}} \times \dcol{\lambda}{\eta} \times \dlevcol{\lambda}{[\eta,\theta)} \right)
    \end{align*}
  The immediate consequence of the above equivalences is that whenever we have a $V[G_\alpha]$-generic filter $G \ast H\subseteq \levcol{\lambda}{\kappa_{\alpha+1}} \ast \dot{\mbb{Q}}$, a further $\lambda$-closed forcing yields a generic $G' \subseteq \levcol{\lambda}{\theta}$ over $V[G_\alpha]$ such that $G \subseteq G'$ and $H \in V[G']$. We will use this fact in a moment below.
    
   We are still working in $V[G_\alpha]$.
   Let us set $W \df V[G_\alpha]$, and let $M \df \ult{W}{\mcl{U}}$.  Let $G'$ be a $V[G_\alpha]$-generic filter on $\levcol{\lambda}{\theta}$ so that $G\subseteq G'$ and $H\in V[G']$, where $G_{\beta}=G_\alpha\ast G\ast H$. We now work in $W[G']$. By the $\delta$-closure of $M$ in $V[G_\alpha]$, the dense embeddings witnessing the forcing equivalences displayed above exist in $M$, so we can recover $H \in M[G']$ using $G'$. Recall that $j_{\mcl U}$ is the canonical ultrapower embedding. By the $\kappa_{\alpha+1}$-c.c. of $\levcol{\lambda}{\kappa_{\alpha+1}}$, we can lift the embedding $j_{\mcl{U}}$ to $j : W[G] \to M[G']$. Furthermore, $\mbb{Q}$ is $\kappa_{\alpha+1}$-directed-closed in $W[G]$, so $j[H] \in M[G']$ is a directed set of size $<\delta<j(\kappa_{\alpha+1})$, and thus it has a lower bound $m \in j(\mbb{Q})$. Forcing below this condition yields a further extension of the embedding to $j : W[G][H] \to M[G'][H']$. The forcing to obtain $H'$ is $j(\kappa_{\alpha+1})$-closed in $M[G']$, which implies it is at least $\lambda$-closed in $W[G']$.

  We are now about to show that $T$ has a non-special subtree of height and size $\lambda$ in $V[G_\beta]$. Since $T = \dot T^{G_\beta}$, we can compute $T$ and $j[T]$ in $M[G']$. In $M[G'][H']$, $j[T]$ is a subtree of $j(T)$ of size $<\delta<j(\kappa_{\alpha+1})$. Since $T$ is non-special and of height $\lambda$ in $V[G_\beta] = W[G][H]$,  \cref{Todorcevic-lemma} implies that it is non-special in $W[G'][H']$ as well. Since $j[T] \cong T$ and being non-special is downwards-absolute, $j[T]$ is a non-special subtree of $j(T)$ in $M[G'][H']$. By elementarity applied to $j$, we get that $T$ has a non-special subtree of height and size $\lambda$ in $V[G_\beta]$. 
\end{proof}

 {\small \textit{Acknowledgments}:    For R. Mohammadpour, this research is part of project No. 2022/47/P/ST1 /00705, co-funded by the National Science Centre and the European Union's Horizon 2020 research and innovation programme under the Marie Sk\l{}odowska-Curie grant agreement No. 945339. M. Eskew was supported by project No. PIN1355423, co-funded by the Austrian Science Fund (FWF) and the Polish National Science Centre (NCN). The authors would like to thank Stevo Todor\v{c}evi\'c for his comments on a draft of this paper.}

\printbibliography[heading=bibintoc]

@incollection{Bau83incollectionIter,
	author = {Baumgartner, James E.},
	booktitle = {Surveys in set theory},
	doi = {10.1017/CBO9780511758867.002},
	mrclass = {03E35 (03E40)},
	mrnumber = {823775},
	mrreviewer = {John K. Truss},
	pages = {1--59},
	publisher = {Cambridge Univ. Press, Cambridge},
	series = {London Math. Soc. Lecture Note Ser.},
	title = {Iterated forcing},
	volume = {87},
	year = {1983},}

@article{BMR,
	author = {Baumgartner, J. and Malitz, J. and Reinhardt, W.},
	doi = {10.1073/pnas.67.4.1748},
	fjournal = {Proceedings of the National Academy of Sciences of the United States of America},
	issn = {0027-8424},
	journal = {Proc. Nat. Acad. Sci. U.S.A.},
	mrclass = {02K05},
	mrnumber = {314621},
	mrreviewer = {Thomas J. Jech},
	pages = {1748--1753},
	title = {Embedding trees in the rationals},
	volume = {67},
	year = {1970},}

@incollection {Todorcevic-handbook,
    AUTHOR = {Todorcevic, S.},
     TITLE = {Trees and linearly ordered sets},
 BOOKTITLE = {Handbook of set-theoretic topology},
     PAGES = {235--293},
 PUBLISHER = {North-Holland, Amsterdam},
      YEAR = {1984},
      ISBN = {0-444-86580-2},
   MRCLASS = {54F05 (04A20 06A05 06A10)},
  MRNUMBER = {776625},
MRREVIEWER = {D.\ J.\ Lutzer},
}

@misc{Todorc-slide,
  author       = {Todorcevic, Stevo},
  title        = {Combinatorial Dichotomies in Set Theory},
  year         = {2024},
  howpublished = {Mostowski Lecture (slides)},
  organization = {University of Wroclaw},
  url          ={https://prac.im.pwr.edu.pl/~twowlc/pub_files/slides/55Todorcevic_slides.pdf},
  note         = {Accessed: 2024-03-24}
}

@article{Tod11articleComb,
	author = {Todorcevic, Stevo},
	doi = {10.2178/bsl/1294186662},
	fjournal = {The Bulletin of Symbolic Logic},
	issn = {1079-8986},
	journal = {Bull. Symbolic Logic},
	mrclass = {03E05 (05D10 06A07)},
	mrnumber = {2760116},
	mrreviewer = {Arnold W. Miller},
	number = {1},
	pages = {1--72},
	title = {Combinatorial dichotomies in set theory},
	volume = {17},
	year = {2011},
	}

@article{Doe13articleRado,
	author = {Doebler, Philipp},
	doi = {10.1142/S0219061313500013},
	fjournal = {Journal of Mathematical Logic},
	issn = {0219-0613},
	journal = {J. Math. Log.},
	mrclass = {03E05 (03E65)},
	mrnumber = {3065118},
	mrreviewer = {A. Kanamori},
	number = {1},
	pages = {1350001, 8},
	title = {Rado's conjecture implies that all stationary set preserving forcings are semiproper},
	%url = {https://doi.org/10.1142/S0219061313500013},
	volume = {13},
	year = {2013},}

@article{Tod83articleOn-a,
	author = {Todorcevic, S.},
	doi = {10.1112/jlms/s2-27.1.1},
	fjournal = {Journal of the London Mathematical Society. Second Series},
	issn = {0024-6107},
	journal = {J. London Math. Soc. (2)},
	mrclass = {03E35 (03E05 03E55 05C15 06A05)},
	mrnumber = {686495},
	number = {1},
	pages = {1--8},
	title = {On a conjecture of {R}. {R}ado},
%	url = {https://doi.org/10.1112/jlms/s2-27.1.1},
	volume = {27},
	year = {1983},
}

@incollection{Tod93incollectionConj,
	author = {Todorcevic, Stevo},
	booktitle = {Finite and infinite combinatorics in sets and logic ({B}anff, {AB}, 1991)},
	mrclass = {03E10 (03C55 06A05)},
	mrnumber = {1261218},
	mrreviewer = {Eva Coplakova},
	pages = {385--398},
	publisher = {Kluwer Acad. Publ., Dordrecht},
	series = {NATO Adv. Sci. Inst. Ser. C: Math. Phys. Sci.},
	title = {Conjectures of {R}ado and {C}hang and cardinal arithmetic},
	%url = {https://mathscinet.ams.org/mathscinet-getitem?mr=1261218},
	volume = {411},
	year = {1993},
}

@article{Tod85articlePart,
	author = {Todorcevic, Stevo},
	date-added = {2024-10-18 13:36:17 +0200},
	date-modified = {2024-10-18 14:21:31 +0200},
	doi = {10.1007/BF02392535},
	fjournal = {Acta Mathematica},
	issn = {0001-5962},
	journal = {Acta Math.},
	mrclass = {03E05 (03E40 04A20 06A10)},
	mrnumber = {793235},
	mrreviewer = {James Baumgartner},
	number = {1-2},
	pages = {1--25},
	title = {Partition relations for partially ordered sets},
%	url = {https://doi.org/10.1007/BF02392535},
	volume = {155},
	year = {1985},}

@article{Rad81articleTheo,
	author = {Rado, Richard},
	date-added = {2025-04-15 08:35:52 +0200},
	date-modified = {2025-04-15 08:49:57 +0200},
	doi = {10.1016/0012-365X(81)90208-9},
	fjournal = {Discrete Mathematics},
	issn = {0012-365X},
	journal = {Discrete Math.},
	mrclass = {05C15 (05C99 06A05)},
	mrnumber = {620672},
	mrreviewer = {Curtis Greene},
	pages = {199--201},
	title = {Theorems on intervals of ordered sets},
%	url = {https://doi.org/10.1016/0012-365X(81)90208-9},
	volume = {35},
	year = {1981},
}

@article{Zha20articleRado,
	author = {Zhang, Jing},
	date-added = {2024-08-07 15:28:35 +0200},
	date-modified = {2024-09-16 13:10:59 +0200},
	doi = {10.1142/S0219061319500156},
	fjournal = {Journal of Mathematical Logic},
	issn = {0219-0613},
	journal = {J. Math. Log.},
	mrclass = {03E05 (03E35 03E55 03E65)},
	mrnumber = {4094551},
	mrreviewer = {Luis Miguel Villegas Silva},
	number = {1},
	pages = {1950015, 35},
	title = {Rado's conjecture and its {B}aire version},
%	url = {https://doi.org/10.1142/S0219061319500156},
	volume = {20},
	year = {2020}
}

@article{Fen99articleRado,
	author = {Feng, Qi},
	date-added = {2025-03-28 11:32:53 +0100},
	date-modified = {2025-03-28 11:35:21 +0100},
	doi = {10.2307/2586748},
	fjournal = {The Journal of Symbolic Logic},
	issn = {0022-4812},
	journal = {J. Symbolic Logic},
	mrclass = {03E05 (03E65)},
	mrnumber = {1683892},
	mrreviewer = {Carlos A. Di Prisco},
	number = {1},
	pages = {38--44},
	title = {Rado's conjecture and presaturation of the nonstationary ideal on {$\omega_1$}},
%	url = {https://doi.org/10.2307/2586748},
	volume = {64},
year = {1999}
}

@book{Jec03bookSet-,
	author = {Jech, Thomas},
	date-added = {2024-07-22 15:37:29 +0200},
	date-modified = {2024-09-16 13:10:59 +0200},
	isbn = {3-540-44085-2},
	mrclass = {03Exx (03-01 03-02)},
	mrnumber = {1940513},
	mrreviewer = {Eva Coplakova},
	note = {The third millennium edition, revised and expanded},
	pages = {xiv+769},
	publisher = {Springer-Verlag, Berlin},
	series = {Springer Monographs in Mathematics},
	title = {Set theory},
	%url = {https://mathscinet.ams.org/mathscinet-getitem?mr=1940513},
	year = {2003}
}

@incollection {foreman,
    AUTHOR = {Foreman, Matthew},
     TITLE = {Ideals and generic elementary embeddings},
 BOOKTITLE = {Handbook of set theory. {V}ols. 1, 2, 3},
     PAGES = {885--1147},
 PUBLISHER = {Springer, Dordrecht},
      YEAR = {2010},
      ISBN = {978-1-4020-4843-2},
   MRCLASS = {03E05 (03E35 03E40 03E55)},
  MRNUMBER = {2768692},
MRREVIEWER = {Sean\ Cox},
       DOI = {10.1007/978-1-4020-5764-9\_14},
     %  URL = {https://doi.org/10.1007/978-1-4020-5764-9_14},
}

@article {fms,
    AUTHOR = {Foreman, M. and Magidor, M. and Shelah, S.},
     TITLE = {Martin's maximum, saturated ideals, and nonregular
              ultrafilters. {I}},
   JOURNAL = {Ann. of Math. (2)},
  FJOURNAL = {Annals of Mathematics. Second Series},
    VOLUME = {127},
      YEAR = {1988},
    NUMBER = {1},
     PAGES = {1--47},
      ISSN = {0003-486X,1939-8980},
   MRCLASS = {03E50 (03E35 03E40 03E55)},
  MRNUMBER = {924672},
MRREVIEWER = {F.\ R.\ Drake},
       DOI = {10.2307/1971415},
      % URL = {https://doi.org/10.2307/1971415},
}

@article {morerigid,
    AUTHOR = {Eskew, Monroe},
     TITLE = {More rigid ideals},
   JOURNAL = {Israel J. Math.},
  FJOURNAL = {Israel Journal of Mathematics},
    VOLUME = {233},
      YEAR = {2019},
    NUMBER = {1},
     PAGES = {225--247},
      ISSN = {0021-2172,1565-8511},
   MRCLASS = {03E55 (03E35)},
  MRNUMBER = {4013973},
MRREVIEWER = {Radek\ Honz\'ik},
       DOI = {10.1007/s11856-019-1905-3},
    %   URL = {https://doi.org/10.1007/s11856-019-1905-3},
}

@article {shioya,
    AUTHOR = {Shioya, Masahiro},
     TITLE = {Easton collapses and a strongly saturated filter},
   JOURNAL = {Arch. Math. Logic},
  FJOURNAL = {Archive for Mathematical Logic},
    VOLUME = {59},
      YEAR = {2020},
    NUMBER = {7-8},
     PAGES = {1027--1036},
      ISSN = {0933-5846,1432-0665},
   MRCLASS = {03E05 (03E35 03E55)},
  MRNUMBER = {4159767},
MRREVIEWER = {Joanna\ Jureczko},
       DOI = {10.1007/s00153-020-00733-8},
    %   URL = {https://doi.org/10.1007/s00153-020-00733-8},
}

\end{document}